\documentclass[a4paper,10pt]{article}
\usepackage{QED}
\usepackage{amsmath}
\usepackage{amssymb}
\usepackage{a4}
\def\a{\alpha}
\def\b{\beta}
\def\g{\gamma}
\def\vp{\varphi}
\def\G{\Gamma}
\def\D{\triangle}

\def\r{\rho}

\def\th{\theta}
\def\l{\lambda}
\def\L{\Lambda}

\def\r{\rho}

\def\ep{\epsilon}
\def\m{\mu}

\def\O{\Omega}
\def\o{\omega}
\def\ep{\varepsilon}
\def\f{\rightarrow}
\def\tr{\triangleright}

\def\v{\vdash}

\def\ou{\vee}
\def\et{\wedge}
\def\<{\langle}
\def\>{\rangle}
\def\F{\displaystyle\frac}
\newtheorem{theorem}{Theorem}[section]
\newtheorem{lemma}{Lemma}[section]

\newtheorem{definition}{Definition}[section]
\newtheorem{notation}{Notation}[section]

\begin{document}
\vspace*{0cm}

\begin{center}
{\Large\bf Strong normalization results by translation}\\[1cm]
\end{center}

\begin{center}
{\bf Ren\'e DAVID and Karim NOUR}\\
LAMA - Equipe LIMD\\
 Universit\'e de Chamb\'ery\\
73376 Le Bourget du Lac\\
e-mail: \{david,nour\}@univ-savoie.fr\\[1cm]
\end{center}

\begin{abstract}
We prove the strong normalization of full classical natural
deduction (i.e. with conjunction, disjunction and permutative
conversions) by using a translation into the simply typed
$\l\m$-calculus. We also extend Mendler's result on recursive
equations to this system.
\end{abstract}

\section{Introduction}

It is well known that, when the underlying logic is the classical one
(i.e. the absurdity rule is allowed) the connectives $\ou$ and $\et$
are redundant (they can be coded by using $\f$ and $\perp$). From a
logical point of view, considering the full logic is thus somehow
useless. However, from the computer science point of view, considering
the full logic is interesting because, by the so-called Curry-Howard
correspondence, formulas can be seen as types for functional
programming languages and correct programs can be extracted from
proofs. The connectives $\et$ and $\ou$ have a functional counter-part
($\et$ corresponds to a product and $\ou$ to a co-product, i.e. a {\em
case of}) and it is thus useful to have them as primitive.

In this paper, we study the typed $\l\m^{\f\et\ou}$-calculus. This
calculus,
 introduced  by de Groote in
  \cite{deGr}, is an
  extension of  Parigot's $\l\m$-calculus. It is the computational counterpart  of classical
natural deduction with $\f$, $\et$ and $\ou$. Three notions of
conversions are necessary in order to have the sub-formula
property : logical, classical and permutative conversions.

The proofs of the strong normalization of the cut-elimination
procedure  for the full classical logic are quite recent and three
kinds of proofs are given in the literature.


{\em Proofs by CPS-translation}.  In \cite{deGr} de Groote also gave
a proof of the strong normalization of the typed
$\l\m^{\f\et\ou}$-calculus using a CPS-translation into the simply
typed $\l$-calculus i.e. the implicative intuitionistic logic but his
proof contains an error as Matthes pointed out in
\cite{Mat1}. Nakazawa and Tatsuta corrected de Groote's proof in
\cite{NaTa} by using the notion of augmentations.


{\em Syntactical proofs}. We gave in \cite{DaNo2} a direct and
syntactical proof of strong norma-lization. The proof is based on a
substitution lemma which stipulates that replacing in a strongly
normalizable deduction an hypothesis by another strongly normali-zable
deduction gives a strongly normalizable deduction. The proof uses a
technical lemma concerning commutative reductions. But, though the
idea of the proof of this lemma (as given in \cite{DaNo2}) works, it
is not complete and (as pointed out by Matthes in a private
communication) it also contains some errors.


{\em Semantical proofs}. K. Saber and the second author gave in
\cite{NaSa} a semantical proof of this result by using the notion of
saturated sets. This proof is a generalization of Parigot's strong
normalization result of the $\lambda\mu$-calculus with the types of
Girard's system ${\cal F}$ by using reducibility candidates. This
proof uses the technical lemma of \cite{DaNo2} concerning commutative
reductions. In \cite{Mat2} and \cite{Tats}, R. Matthes and Tastuta
give another semantical proofs by using a (more complex) concept of
saturated sets.

\medskip

This paper presents  a new proof of the strong normalization of
the simply typed $\l\m^{\f\et\ou}$-calculus. This proof is
formalizable in Peano first order arithmetic and does not need any
complex lemma. It is obtained by giving a translation of this
calculus into the $\l\m$-calculus. The coding of $\et$ and $\ou$
in classical logic is the usual one but, as far as we know, the
fact that this coding behaves correctly with the computation, via
the Curry-Howard correspondence, has never been analyzed. This
proof is much simpler than the existing  ones\footnote{Recently,
we have been aware of a paper by Wojdyga \cite{woj} who uses the
same kind of translations but where all the atomic types are
collapsed to $\bot$. Our translation allows us to extend trivially
Mendler's result whereas the one of Wojdyga, of course, does
not.}.

It also presents a new result. Mendler \cite{Mend2} has shown that
strong normalization is preserved if, on types, we allow some
equations satisfying natural (and necessary) conditions.  Mendler's
result concerned the implicative fragment of intuitionistic logic.  By
using the previous translation, we extend here this result to full
classical logic .

The paper  is organized as follows. Section \ref{2} gives the
various systems for which we prove the strong normalization.
Section \ref{6} gives the translation of the
$\l\m^{\f\et\ou}$-calculus into the $\l\m$-calculus and section
\ref{7} extends Mendler's theorem to the
$\l\m^{\f\et\ou}$-calculus. For a first reading, sections \ref{3},
\ref{4} and \ref{rab} may be skipped. They have been added to have
complete proofs of the other results. Section \ref{3} contains the
proof, by the first author, of the the strong normalization of the
simply typed $\l$-calculus. Section \ref{4} gives a translation of
the $\l\m$-calculus into the $\l$-calculus and section \ref{rab}
gives some well known properties of the $\l\m$-calculus. Finally,
the appendix gives a detailed proof of a lemma that needs a long
but easy case analysis.

\section{The systems}\label{2}

\begin{definition}
Let ${\cal V}$ and ${\cal W}$ be disjoint sets of variables.
\begin{enumerate}
  \item The set of $\l$-terms is defined by the following grammar

$${\cal M} := {\cal V} \ | \ \l {\cal V}.  {\cal M} \ | \ ({\cal M} \; {\cal M})$$

  \item The set of $\l\m$-terms is defined by the following grammar

$${\cal M}' := {\cal V}  \ | \ \l {\cal V}.  {\cal M}' \ | \ ({\cal
    M}' \; {\cal M}') \ |  \ \m {\cal W}. {\cal M}' \ |  \ ({\cal W} \; {\cal M}')$$

\item The set of $\l\m^{\f\et\ou}$-terms is defined by the following grammar

$${\cal M}'' ::= {\cal V} \ | \ \l {\cal V}. {\cal M}'' \ | \ ({\cal
M}'' \; {\cal E}) \ | \ \<{\cal M}'' , {\cal M}'' \> \ | \ \o_1 {\cal
M}'' \ | \ \o_2 {\cal M}'' \ | \ \mu {\cal W}. {\cal M}'' \ | \ ({\cal
W} \; {\cal M}'')$$

$${\cal E} ::=  {\cal M}'' \ | \  \pi_1 \ | \ \pi_2 \ | \ [{\cal V}.{\cal M}'' ,{\cal V}.{\cal M}'']$$

\end{enumerate}
\end{definition}

Note that, for the $\l\m$-calculus, we have adopted here the
so-called de Groote calculus which is the extension of Parigot's
calculus where the distinction between named and un-named terms is
forgotten. In this calculus, $\m\a$ is not necessarily followed by
$[\b]$. We also write $(\a \; M)$ instead of $[\a] M$.

\begin{definition}\label{def_red}

\begin{enumerate}
  \item The  reduction rule for the $\l$-calculus is the $\b$-rule.

  $$(\l x. M \; N) \tr_{\b} M[x:=N]$$
  \item The  reduction rules for the $\l\m$-calculus are the
  $\b$-rule and the $\m$-rule

  $$(\mu \a. M \; N) \tr_{\m}  \mu \a. M[(\a \;
  L) := (\a \; (L \; N))]$$

  \item The  reduction rules for the $\l\m^{\f\et\ou}$-calculus
  are those of the $\l\m$-calculus  together with the following rules

\begin{center}
$(\< M_1,M_2 \> \; \pi_i) \tr M_i$

 $(\o_i M\; [x_1. N_1 ,x_2. N_2]) \tr N_i[x_i:=M]$

 $(M \; [x_1.N_1 , x_2.N_2] \; \ep)
\tr (M \; [x_1.(N_1 \; \ep) , x_2.(N_2 \; \ep)])$

$(\mu \a. M \; \ep) \tr \mu \a. M[(\a \; N):= (\a \; (N \; \ep))]$
\end{center}
\end{enumerate}
\end{definition}

\begin{definition}
Let ${\cal A}$ be a set of atomic constants.
\begin{enumerate}

\item  The set ${\cal T}$ of types is defined by the following grammar

$${\cal T} ::= \; {\cal A}\cup \{\bot\} \; \mid
{\cal T} \f {\cal T}$$

\item The set ${\cal T}'$ of types is defined by the following grammar

$${\cal T}' ::= \; {\cal A}\cup \{\bot\} \; \mid \;   {\cal T}' \f
{\cal T}' \; \mid \; {\cal T}' \et {\cal T}' \; \mid \; {\cal T}'
\ou {\cal T}' $$

\end{enumerate}
As usual $\neg A$ is an abbreviation for $A \f \bot$.

\end{definition}

\begin{definition}
\begin{enumerate}
  \item A $\l$-context is a set of
declarations of the form $x : A$ where $x \in {\cal V}$, $A \in
{\cal T}$ and where a variable may occur at most once.
  \item A $\l\m$-context is a set of
declarations of the form $x : A$ or $\a : \neg B$ where $x \in
{\cal V}$, $\a \in {\cal W}$, $A,B \in {\cal T}$ and where a
variable may occur at most once.
  \item A $\l\m^{\f\et\ou}$-context is a set of
declarations of the form $x : A$ or $\a : \neg B$ where $x \in
{\cal V}$, $\a \in {\cal W}$, $A,B \in {\cal T}'$ and where a
variable may occur at most once.
\end{enumerate}

\end{definition}

\begin{definition}
\begin{enumerate}
\item The simply typed $\l$-calculus (denoted ${\cal S}$) is defined
by the following typing rules where $\G$ is a $\l$-context,

\begin{center}

$\F{}{\G , x : A \v x : A} \, ax$ \hspace{0.5cm} $\F{\G, x: A \v M
: B} {\G \v \l
x.M : A \f B} \, \f_i$\\[0.5cm]

$\F{\G \v M : A \f B \quad \G \v N : A} {\G \v (M \; N)
: B }\, \f_e$
\end{center}
\item The simply typed $\l\m$-calculus (denoted ${\cal S}^{\m}$) is
obtained by adding to the previous rules (where $\G$ now is a
$\l\m$-context) the following rules.

\begin{center}
$\F{\G , \a : \neg A \v M : A} {\G , \a : \neg A\v (\a \; M) :
\bot} \bot_i$\hspace{0.5cm} $\F{\G , \a : \neg A \v M : \bot} {\G
\v \mu \a.M : A} \bot_e$
\end{center}
  \item The simply typed $\l\m^{\f\et\ou}$-calculus (denoted ${\cal
S}^{\f\et\ou}$) is defined by adding to the previous rules (where
$\G$ now is a $\l\m^{\f\et\ou}$-context) the following rules.

\begin{center}

$\F{\G \v M : A_1 \quad \G \v N : A_2} {\G \v \< M,N \>
: A_1 \et A_2} \, \et_i$ \hspace{0.5cm} $\F{\G \v M : A_1 \et A_2}
{\G \v  (M \; \pi_i): A_i } \, \et_e$
\medskip

$\F{\G \v M :
A_j} {\G \v \o_j M : A_1 \ou A_2} \, \ou_i$
\medskip

$\F{\G \v M : A_1 \ou A_2 \quad \G , x_1 : A_1 \v N_1 : C \quad
\G , x_2 : A_2 \v N_2 : C} {\G \v (M\;[x_1.N_1,
x_2.N_2]) : C} \, \ou_e$
\medskip

\end{center}

\item If $\approx$ is a congruence on ${\cal T}$ (resp. ${\cal T}'$),
we define the systems ${\cal S}_{\approx}$, (resp. ${\cal
S}^{\m}_{\approx}$, ${\cal S}^{\f\et\ou}_{\approx}$) as the system
 ${\cal S}$ (resp.  ${\cal S}^{\m}$,  ${\cal
S}^{\f\et\ou}$) where we have added the following typing rule.

$$\F{\G \v M : A \;\;\; A \approx B}{\G \v M : B} \, \approx$$

\end{enumerate}
\end{definition}

\begin{notation}
\begin{itemize}

\item We will denote by $size(M)$ the complexity of the term $M$.

\item Let $\overrightarrow{P}$ be a finite (possibly empty) sequence
of terms and $M$ be a term.  We denote by $(M \ \overrightarrow{P})$
the term $(M \ P_1 \ ... \ P_n)$ where $\overrightarrow{P}= P_1, ...,
P_n$.

\item In the rest of the paper $\tr$ will represent the reduction
determined by all the rules of the corresponding calculus.

\item If we want to consider only some of the rules we will mention
them as a subscript of $\tr$. For example, in the
$\l\m^{\f\et\ou}$-calculus, $M \tr_{\b\m} N$ means that $M$ reduces
to $N$ either by the $\b$-rule or by the $\m$-rule.

\item As usual, $\tr_r^*$ (resp. $\tr_r^+$) denotes the symmetric and
transitive closure of $\tr_r$ (resp. the transitive closure of
$\tr_r$). We denote $M \tr_{r}^1 N$ iff $M = N$ or $M \tr_r N$.

\item A term $M$ is strongly normalizable for a reduction $\tr_r$
(denoted as $M \in SN_r$) if there is no infinite sequence of
reductions $\tr_r$ starting from $M$. For $M \in SN_r$, we denote by
$\eta_r(M)$ the length of the longest reduction of $M$.

\item If $M \tr_r^* N$, we denote by $lg(M \tr_r^* N)$ the number of
steps in the reduction $M \tr_r^* N$. If  $M \tr^* N$, we denote
by $lg_r(M \tr^* N)$ the number of $\tr_r$ steps of the reduction
 in $M \tr^* N$.

\end{itemize}
\end{notation}

\section{Strong normalization of $\cal{S}$}\label{3}

This section gives a simple proof (due to the first author) of the
strong normalization of the simply typed $\l$-calculus.

\begin{lemma} \label{a}
Let $M,N, \overrightarrow{O} \in {\cal M}$. If $M,N,\overrightarrow{O}
\in SN_{\b}$ and $(M\;N\;\overrightarrow{O} )\not\in SN_{\b}$, then
$(M_{1}[x:=N]\;\overrightarrow{O})\not\in SN_{\b}$ for some $M_1$ such that
$M \tr_{\b}^* \l x.M_{1}$.
\end{lemma}

\begin{proof}
Since $M,N,\overrightarrow{O}\in SN_{\b}$, the infinite reduction of
$T=(M\;N\; \overrightarrow{O})$ looks like: $T\tr_{\b}^*(\l
x.M_{1}\;N_{1} \ \overrightarrow{O_{1}}\;)\tr_{\b}(M_{1}[x:=N_{1}]\;
\overrightarrow{O_{1}})\tr_{\b}^*\ldots $. The result immediately
follows from the fact that $(M_{1}[x:=N]\;\overrightarrow{O}
)\tr_{\b}^*(M_{1}[x:=N_{1}]\;\overrightarrow{O_{1}})$.
\end{proof}

\begin{lemma}\label{b}
If $M,N \in SN_{\b}$ are typed $\l$-terms, then $M[x:=N] \in SN_{\b}$.
\end{lemma}

\begin{proof}
By induction on $(type(N),\eta_{\b}(M),size(M))$ where $type(N)$ is
the complexity of the type of $N$.  The cases $M=\l x.M_1$ and $M=(y\;
\overrightarrow{O})$ for $y \neq x$ are trivial.

\begin{itemize}
\item $M=(\lambda y.P\;Q\;\overrightarrow{O})$. By the induction
hypothesis, $P[x:=N],Q[x:=N]$ and $\overrightarrow{O}[x:=N]$ are
in $SN_{\b}$. By lemma \ref{a} it is enough to show that
$(P[x:=N][y:=Q[x:=N]]\; \overrightarrow{O}[x:=N])=M^{\prime
}[x:=N] \in SN_{\b}$ where $M^{\prime
}=(P[y:=Q]\;\overrightarrow{O})$. But $\eta_{\b}(M^{\prime
})<\eta_{\b}(M)$ and the result follows from the induction
hypothesis.

\item $M=(x\;P\;\overrightarrow{O})$. By the induction hypothesis,
$P_{1}=P[x:=N]$ and $\overrightarrow{O_{1}}=\overrightarrow{O}[x:=N]$
are in $SN_{\b}$. By lemma \ref{a} it is enough to show that if $N
\tr_{\b}^*\l y.N_{1}$ then $M_{1}=(N_{1}[y:=P_{1}]\;
\overrightarrow{O_{1}}) \in SN_{\b}$. By the induction hypothesis
(since $ type(P_{1})<type(N)$) $N_{1}[y:=P_{1}] \in SN_{\b}$ and thus,
by the induction hypothesis (since
$M_{1}=(z\;\overrightarrow{O_{1\;}})\;[z:=N_{1}[y:=P_{1}]]$ and
$type(N_{1})<type(N)$) $M_{1} \in SN_{\b}$.
\end{itemize}
\end{proof}

\begin{theorem}\label{sn_rd}
The simply typed $\l$-calculus is strongly normalizing.
\end{theorem}
\begin{proof}
By induction on $M$. The cases $M=x$ or $ M=\lambda x.P$ are
trivial. If $M=(N \ P)=(z\;P)[z:=N]$ this  follows from  lemma
\ref{b} and the induction hypothesis.
\end{proof}

\section{A translation of the $\l\m$-calculus into the
$\l$-calculus}\label{4}

We give here a translation of the simply typed $\l\m$-calculus into
the simply typed $\l$-calculus. This translation is a simplified
version of Parigot's translation in \cite{Par2}. His translation uses
both a translation of types (by replacing each atomic formula $A$ by
$\neg\neg A$) and a translation of terms.
But it is known that, in the implicative fragment of propositional
logic, it is enough to add $\neg\neg$ in front of the rightmost
variable. The translation we have chosen consists in decomposing
the formulas (by using the terms $T_A$) until the rightmost
variable is found and then using the constants $c_X$ of type
$\neg\neg X \f X$. With such a translation the type does not
change.

Since the translation of a term of the form $\m\a.M$ uses the type of
$\a$, a formal presentation of this translation would need the use of
$\l$-calculus and $\l\m$-calculus\\ \`a la Church. For simplicity of
notations we have kept a presentation \`a la Curry, mentioning the
types only when it is necessary.

We extend the system ${\cal S}$ by adding, for each propositional
variable $X$,  a constant $c_X$. When  the constants that occur in
a term $M$ are $c_{X_1},...,c_{X_n}$,  the notation $\G \v_{{\cal
S}^c} M : A$ will mean $\G,c_{X_1}:\neg \neg X_1 \f
X_1,...,c_{X_n}:\neg \neg X_n \f X_n \v_{{\cal S}} M : A$.

\begin{definition}
 For every $A \in {\cal T}$, we define a $\l$-term $T_A$ as follows:
\begin{itemize}
\item $T_{\bot} = \l x. (x \; \l y.y)$
\item $T_X = c_X$
\item $T_{A \f B} = \l x.\l y. (T_B \; \l u.(x \; \l v. (u \; (v \;
y))))$
\end{itemize}
\end{definition}

\begin{lemma}\label{tran}
For every $A \in {\cal T}$, $\v_{{\cal S}^c} T_A : \neg \neg A \f A$.
\end{lemma}

\begin{proof}
By induction on $A$.
\end{proof}

\begin{definition}
\begin{enumerate}
\item We associate to each $\m$-variable $\a$ of type ${\neg A}$ a $\l$-variable $x_{\a}$ of type $\neg A$.
\item A typed $\l\m$-term $M$ is translated into an $\l$-term $M^{\diamond}$ as follows:
\begin{itemize}
\item $\{x\}^{\diamond} = x$
\item $\{\l x. M\}^{\diamond} = \l x. M^{\diamond}$
\item $\{(M \; N)\}^{\diamond} = (M^{\diamond} \; N^{\diamond})$
\item $\{\m \a. M\}^{\diamond} = (T_A \;\; \l x_{\a}. M^{\diamond})$ if the type of $\a$ is $\neg A$
\item $\{(\a \; M)\}^{\diamond} = (x_{\a} \; M^{\diamond})$
\end{itemize}
\end{enumerate}
\end{definition}

\begin{lemma}
\begin{enumerate}
\item $M^{\diamond}[x :=N^{\diamond}] = \{M[x:=N]\}^{\diamond}$.
\item $M^{\diamond}[x_{\a} := \l v. (x_{\a} \; (v \; N^{\diamond}))]
 \tr_{\beta}^* \{M[(\a \; L) := (\a \; (L \; N))]\}^{\diamond}$.
\end{enumerate}
\end{lemma}

\begin{proof}
By induction on $M$. The first point is immediate.  For the
second, the only interesting case is $M = (\a \; K)$.  Then,
$M^{\diamond}[x_{\a} := \l v. (x_{\a} \; (v \; N^{\diamond}))] =(
\l v. (x_{\a} \; (v \; N^{\diamond})) \ K^{\diamond}[x_{\a} := \l
v. (x_{\a} \; (v \; N^{\diamond}))]) $ $\tr_{\b} \ (x_{\a}
\;(K^{\diamond}[x_{\a} := \l v. (x_{\a} \; (v \; N^{\diamond}))]
\; N^{\diamond}) \tr_{\b}^*(x_{\a} \;(\{K[(\a \; L) := (\a \; (L
\; N))]\}^{\diamond} \; N^{\diamond})= \{M[(\a\; L) := (\a \; (L
\; N))]\}^{\diamond}$.
\end{proof}

\begin{lemma}\label{simulation1}
Let $M \in {\cal M}'$.
\begin{enumerate}
\item If $M \tr_{\b} N$, then $M^{\diamond} \tr_{\b}^+ N^{\diamond}$.
\item If $M \tr_{\m} N$, then $M^{\diamond} \tr_{\b}^+ N^{\diamond}$.
\item If $M \tr_{\b\m}^* N$, then $M^{\diamond} \tr_{\b}^* N^{\diamond}$ and
$lg(M^{\diamond} \tr_{\b}^* N^{\diamond})
\geq lg(M \tr_{\b\m}^* N)$.
\end{enumerate}
\end{lemma}

\begin{proof}
By induction on $M$. (1) is immediate. (2) is as follows.

\medskip

\noindent $(\mu \a^{\neg(A \f B)}.M \; N) \tr_{\m} \mu \a^{\neg
B}.M[(\a^{\neg(A \f B)} \; L) := (\a^{\neg B} \; (L \; N))]$ is
translated by \\ $\{(\mu \a.M \; N)\}^{\diamond} = (T_{A \f B} \;\; \l
x_{\a}.M^{\diamond} \; N^{\circ}) \ \tr_{\b}^+ (T_B \; \l
u.M^{\diamond}[x_{\a} := \l v. (u \; (v \; N^{\diamond}))] = (T_B \;
\l x_{\a}.M^{\diamond}[x_{\a} := \l v. (x_{\a} \; (v \;
N^{\diamond}))] \ \tr_{\beta}^* (T_B \; \l x_{\a}.\{M[(\a \; L) := (\a
\; (L \; N))]\}^{\diamond}) = \{\m \a.M[(\a \; L) := (\a \; (L \;
N))]\}^{\diamond}$.

  \medskip

\noindent  (3) follows immediately from (1) and (2).
\end{proof}

\begin{lemma}\label{sn1->sn2}
Let $M \in {\cal M}'$. If $M^{\diamond} \in SN_{\b}$, then $M \in
SN_{\b\m}$.
\end{lemma}

\begin{proof}
Let $n=\eta_{\b}(M^{\diamond})+1$. If $M \not \in SN_{\b\m}$,
there is $N$ such that $M \tr_{\b\m}^* N$ and $lg(M \tr_{\b\m}^*
N) \geq n$. Thus, by lemma \ref{simulation1}, $M^{\diamond}
\tr_{\b}^* N^{\diamond}$ and $lg(M^{\diamond} \tr_{\b}^*
N^{\diamond}) \geq lg(M \tr_{\b\m}^* N) \geq
\eta_{\b}(M^{\diamond})+1$. This contradicts the definition of
$\eta_{\b}(M^{\diamond})$.
\end{proof}

\begin{lemma}\label{coding1}
If $\G \v_{{\cal S}^{\m}} M : A$, then $\G^{\diamond} \v_{{\cal S}^c}
M^{\diamond} : A$ where $\G^{\diamond}$ is obtained from $\G$ by
replacing $\a : \neg B$ by ${x}_{\a} : \neg B$.
\end{lemma}

\begin{proof}
By induction on the typing $\G \v_{{\cal S}^{\m}} M : A$. Use lemma \ref{tran}.
\end{proof}

\begin{theorem}\label{SN_M}
The simply typed $\l\m$-calculus is strongly normalizing for
$\tr_{\b\m}$.
\end{theorem}

\begin{proof}
A consequence of lemmas  \ref{sn1->sn2}, \ref{coding1} and theorem
\ref{sn_rd}.
\end{proof}

\section{Some classical results on the
$\l\m$-calculus}\label{rab}

The translation given in the next section needs the addition, to
the $\l\m$-calculus, of the following reductions rules.

\begin{center}
$(\b \; \mu \a. M) \tr_{\r} M[\a := \b]$

 $\mu \a. (\a \; M) \tr_{\th} M$ if $\a \not \in Fv(M)$
\end{center}

We will need some classical results about these new rules. For the
paper to remain self-contained, we also have added their proofs.
The reader who already knows these results or is only interested
by the results of the next section may skip this part.

\subsection{Adding $\tr_{\r\th}$ does not change $SN$}

\begin{theorem}\label{retardemant}
Let $M \in {\cal M}'$ be such that $M \in SN_{\b\m}$. Then $M \in
SN_{\b\m\r\th}$.
\end{theorem}

\begin{proof}
This follows from the fact that $\tr_{\r\th}$ can be  postponed
(theorem \ref{ch1:pprt} below) and
 that $\tr_{\r\th}$ is strongly normalizing (lemma
\ref{SNrth} below).
\end{proof}

\begin{lemma}\label{SNrth}
The reduction $\tr_{\r\th}$ is strongly normalizing.
\end{lemma}

\begin{proof}
The reduction  $\tr_{\r\th}$ decreases the size.
\end{proof}

\begin{theorem}\label{ch1:pprt} Let $M,N$ be such that $M\tr^*_{\b \m
\rho \theta}N$ and $lg_{\b\m}(M\tr^*_{\b \m \rho \theta} N) \geq
1$. Then $M\tr_{\b \m }^+P\tr^*_{\rho \theta}N$ for some $P$.
\end{theorem}

This is proved in two steps. First we show that the
$\tr_{\th}$-reduction can be postponed w.r.t. to $\tr_{\b\m\rho}$
(theorem \ref{ch1:ppt}). Then we show that the $\tr_{\rho}$-rule
can be postponed w.r.t. the remaining rules (theorem
\ref{ch1:ppr}).

\begin{definition}
 Say that  $P \tr_{\mu_0}P' $  if $P= (\mu \a M\;N),
  P'=\mu \a M[(\a \ L]:=(\a \ (L \ N))]$ and $\a$ occurs at most once in $M$

\end{definition}

\begin{lemma}
\label{ch1:bth}

\begin{enumerate}
  \item Assume $M\tr_{\th}P\tr_{\b\m}N$. Then either $M\tr_{\b\m}Q\tr^*_{\th}N$ for some $Q$  or
$M\tr_{\mu_0}R\tr_{\b\m}Q\tr_{\th}N$ for some $R$, $Q$.

  \item Let $M\tr_{\th}P\tr_{\mu_0}N$. Then either
$M\tr_{\mu_0}Q\tr_{\th}N$ for some $Q$ or
$M\tr_{\mu_0}R\tr_{\mu_0}Q\tr_{\th}N$ for some $R$, $Q$.

\item Let $M\tr_{\th}P\tr_{\rho}N$. Then
$M\tr_{\rho}Q\tr_{\th}N$.
\end{enumerate}
\end{lemma}

\begin{proof}
By induction on $M$.
\end{proof}

\begin{lemma} \label{ch1:pth}

Let $M\tr^*_{\th}P\tr_{\mu_0}N$. Then, $
M\tr^*_{\mu_0}Q\tr^*_{\th}N$ for some $Q$ such that
$lg(M\tr_{\th}^*P)=lg(Q\tr_{\th}^*N)$.

\end{lemma}

\begin{proof}
By induction on $lg(M\tr_{\th}^*P)$.
\end{proof}

\begin{theorem} \label{ch1:ppt} Let $M\tr^*_{\th}P\tr_{\b \m  \rho}N$.
Then, $ M\tr_{\b \m  \rho}^+Q\tr^*_{\th}N$ for some $Q$.
\end{theorem}

\begin{proof}
By induction on $lg(M\tr_{\th}^*P)$.
\end{proof}

\begin{lemma} \label{ch1:brh}
\begin{enumerate}
\item Let $M\tr_{\rho}P\tr_{\b}N$. Then $M\tr_{\b}Q\tr^*_{\rho}N$ for some
$Q$.
\item Let $M$, $M'$, $N$ be such that
$M\tr_{\rho}M'$ and $\a \notin Fv(N)$. Then either $M[(\a \
L]:=(\a \ (L \ N))]\tr_{\rho}M'[(\a \ L]:=(\a \ (L \ N))]$  or
$M[(\a \ L]:=(\a \ (L \ N))]\tr_{\mu}P\tr_{\rho}M'[(\a \ L]:=(\a \
(L \ N))]$  for some $P$.
\item Let $M\tr_{\rho}P\tr_{\mu}N$. Then
$ M\tr_{\mu}Q\tr^*_{\rho}N$ for some $Q$.
\end{enumerate}
\end{lemma}

\begin{proof}
By induction on $M$.
\end{proof}

\begin{theorem} \label{ch1:ppr}
Let $M\tr^*_{\rho}P\tr_{\b \m}N$. Then $M\tr_{\b \m}Q\tr^*_{\rho}N$ for some $Q$.
\end{theorem}

\begin{proof}
By induction on $lg(M\tr_{\rho}^*P)$.
\end{proof}

\subsection{Commutation lemmas}

The goal of this section is lemma \ref{diag*} below. Its proof
necessitates some preliminary lemmas.

\begin{lemma}\label{diag}
\begin{enumerate}
\item If $M \tr_{\r} P$ and $M \tr_{\r\th} Q$, then $P = Q$ or $P
\tr_{\r\th} N$ and $Q \tr_{\r} N$ for some $N$.
\item If $M \tr_{\r} P$ and $M \tr_{\b\m} Q$, then $P \tr_{\b\m} N$
and $Q \tr_{\r}^* N$ for some $N$.
\end{enumerate}
\end{lemma}

\begin{proof}
 By simple case analysis.
\end{proof}

\begin{lemma}\label{new}
\begin{enumerate}
\item If $M \tr_{\r}^* P$ and $M \tr_{{\r\th}^1} Q$, then $P
\tr_{{\r\th}^1} N$ and $Q \tr_{\r}^* N$ for some $N$.
\item If $M \tr_{\r}^* P$ and $M \tr_{\r\th}^* Q$, then $P
\tr_{\r\th}^* N$ and $Q \tr_{\r}^* N$ for some $N$.
\item If $M \tr_{\r}^* P$ and $M \tr_{\b\m} Q$, then $P \tr_{\b\m} N$
and $Q \tr_{\r}^* N$ for some $N$.
\end{enumerate}
\end{lemma}

\begin{proof}
\begin{enumerate}
\item By induction on $\eta_{\r}(M)$. Use (1) of lemma \ref{diag}.
\item By induction on $lg(M \tr_{\r\th}^* Q)$. Use (1).
\item By induction on $\eta_{\r}(M)$. Use (2) of lemma \ref{diag}.
\end{enumerate}
\end{proof}

\begin{lemma}\label{diag*}
If $M \tr_{\r}^* P$ and $M \tr_{\b\m\r\th}^* Q$, then $P
\tr_{\b\m\r\th}^* N$, $Q \tr_{\r}^* N$ for some $N$ and
$lg_{\b\m}(P \tr_{\b\m\r\th}^* N) = lg_{\b\m}(M \tr_{\b\m\r\th}^*
Q)$.

\end{lemma}

\begin{proof}
By induction on $lg_{\b\m}(M \tr_{\b\m\r\th}^* Q)$. If $M
\tr_{\b\m\r\th}^* M_1 \tr_{\b\m} M_2 \tr_{\r\th}^* Q$, then, by
induction hypothesis, $P \tr_{\b\m\r\th}^* N_1$, $M_1 \tr_{\r}^*
N_1$ and $lg_{\b\m}(P \tr_{\b\m\r\th}^* N_1) = lg_{\b\m}(M \tr^*
M_1)$. By (3) of lemma \ref{new}, $N_1 \tr_{\b\m} N_2$ and $M_2
\tr_{\r}^* N_2$ for some $N_2$. And finally, by (2) of lemma
\ref{new}, $N_2 \tr_{\r\th}^* N$ and $Q \tr_{\r}^* N$ for some
$N$. Thus $P \tr_{\b\m\r\th}^* N$, $Q \tr_{\r}^* N$ and
$lg_{\b\m}(P \tr_{\b\m\r\th}^* N) = lg_{\b\m}(M \tr_{\b\m\r\th}^*
Q)$.
\end{proof}

\section{A translation of the $\l\m^{\f\et\ou}$-calculus into the
$\l\m$-calculus}\label{6}

We code  $\et$ and $\ou$ by their usual equivalent (using $\f$ and
$\bot$) in classical logic.

\begin{definition}
We define the translation $A^{\circ} \in \cal{T}$ of a type $A \in \cal{T}'$ by
induction on $A$
 as follows.
\begin{itemize}
\item $\{A\}^{\circ}= A$ for $A \in {\cal A}\cup \{\bot\}$
\item $\{A_1 \f A_2\}^{\circ}= A_1^{\circ} \f A_2^{\circ}$
\item $\{A_1 \et A_2\}^{\circ}= \neg (A_1^{\circ} \f (A_2^{\circ} \f
  \bot))$
\item $\{A_1 \ou A_2\}^{\circ}= \neg A_1^{\circ} \f (\neg A_2^{\circ} \f
  \bot)$
\end{itemize}
\end{definition}

\begin{lemma}
For every $A \in {\cal T}'$, $A^{\circ}$ is classically equivalent to
$A$.
\end{lemma}

\begin{proof}
By induction on  $A$.
\end{proof}

\begin{definition}
Let $\varphi$ a special $\m$-variable. A term $M \in {\cal M}''$ is
translated into a $\l\m$-term $M^{\circ}$ as follows:
\begin{itemize}
\item $\{x\}^{\circ}= x$
\item $\{\l x. M\}^{\circ}= \l x. M^{\circ}$
\item $\{(M \; N)\}^{\circ}= (M^{\circ} \; N^{\circ})$
\item $\{\m \a. M\}^{\circ} = \m \a.  M^{\circ}$
\item $\{(\a \; M)\}^{\circ} = (\a \; M^{\circ})$
\item $\{\<M ,N\>\}^{\circ}= \l x. (x \; M^{\circ} \; N^{\circ})$
\item $\{M \pi_i\}^{\circ}= \m \a. (\vp \; (M^{\circ} \; \l x_1.\l x_2.
  \m \g. (\a \; x_i)))$ where $\g$ is a fresh variable
\item $\{M \; [x_1.N_1 , x_2.N_2]\}^{\circ}= \m \a. (\vp \; (M^{\circ}
  \; \l x_1. \m \g. (\a \; N_1^{\circ}) \; \l x_2. \m \g. (\a \;
  N_2^{\circ})))$ where $\g$ is a fresh variable
\item $\{\o_i M\}^{\circ}= \l x_1. \l x_2. (x_i \; M^{\circ})$
\end{itemize}
\end{definition}

\noindent {\bf Remarks}
\begin{itemize}
\item The introduction of the free variable $\vp$ in the definition of
$\{M \; [x_1.N_1 , x_2.N_2]\}^{\circ}$ and $\{M \pi_i\}^{\circ}$ is
not necessary for lemma \ref{coding2}. The reason of this introduction
is that, otherwise, to simulate the reductions of the
$\l\m^{\f\et\ou}$-calculus we would have to introduce new reductions
rules for the $\l\m$-calculus and thus to prove $SN$ of this extension
whereas, using $\vp$, the simulation is done with the usual rules of
the $\l\m$-calculus.
  \item There is another  way of coding $\et$ and $\ou$  by using
intuitionistic second order logic.

\begin{itemize}
\item $\{A_1 \et A_2\}^{\circ}= \forall X ((A_1^{\circ} \f
  (A_2^{\circ} \f X)) \f X)$
\item $\{A_1 \ou A_2\}^{\circ}= \forall X ((A_1^{\circ} \f X) \f
  ((A_2^{\circ} \f X) \f X))$
\end{itemize}

The translation of  $\{\<M ,N\>\}^{\circ}$ and $\{\o_i
M\}^{\circ}$ are the same but the translation of $\{M
\pi_i\}^{\circ}$ will be $(M^{\circ} \; \l x_1.\l x_2.x_i)$ and
the one of $\{M \; [x_1.N_1 , x_2.N_2]\}^{\circ}$ would be $
(M^{\circ} \; \l x_1.N_1^{\circ} \; \l x_2.N_2^{\circ})$. But it
is easily checked that the permutative conversions are not
correctly simulated by this translation whereas, in our
translation, they are.
\item Finally  note that, as given in
definition \ref{def_red},  the reduction rules for the
$\l\m^{\f\et\ou}$-calculus do not include $\tr_{\r}$ and
$\tr_{\th}$. We could have added them  and the given translation
 would have worked in a similar way. We decided not to do so
(although these rules were already considered by Parigot) because
they, usually, are not included neither in the $\l\m$-calculus nor
in the $\l\m^{\f\et\ou}$-calculus. Moreover some of the lemma
given below would need a bit more complex statement.

\end{itemize}

\begin{lemma}
\begin{enumerate}
\item $\{M[x:=N]\}^{\circ} = M^{\circ}[x:=N^{\circ}]$.
\item $\{M[(\a \; N):= (\a \; (N \; \ep))]\}^{\circ} = M^{\circ}[(\a \;
  N^{\circ}):= (\a \; \{(N \; \ep)\}^{\circ})]$.
\end{enumerate}
\end{lemma}

\begin{proof}
By induction on $M$.
\end{proof}

\begin{lemma}\label{coding2}
If $\G \v_{{\cal S}^{\f\et\ou}} M : A$, then $\G^{\circ} \v_{{\cal
S}^{\m}} M^{\circ} : A^{\circ}$ where $\G^{\circ}$ is obtained
from $\G$ by replacing  all the types by their translations and by
declaring $ \vp$ of type $ \neg \bot$.
\end{lemma}

\begin{proof}
By induction on a derivation of $\G \v_{{\cal S}^{\f\et\ou}} M : A$.
\end{proof}

\begin{lemma}\label{simulation2}
Let $M \in {\cal M}''$. If $M \tr N$, then there is $P \in {\cal M}'$
such that $M^{\circ} \tr_{\b\m\r\th}^* P$, $N^{\circ} \tr_{\r}^*P$ and
$lg_{\b\m}(M^{\circ} \tr_{\b\m\r\th}^* P) \geq 1$.
\end{lemma}

\begin{proof}
By case analysis. The details are given in the appendix, section
\ref{app}.
\end{proof}

\begin{lemma} \label{simulation3}
Let $M \in {\cal M}''$. If $M \tr^* N$, then there is $P\in {\cal M}'$
such that $M^{\circ} \tr_{\b\m\r\th}^* P$, $N^{\circ} \tr_{\r}^* P$
and $lg_{\b\m}(M^{\circ} \tr_{\b\m\r\th}^* P) \geq lg(M \tr^* N)$.
\end{lemma}

\begin{proof}
By induction on $lg(M \tr^* N)$. If $M \tr^* L \tr N$, then, by
induction hypothesis, there is $Q\in {\cal M}'$ such that $M^{\circ}
\tr_{\b\m\r\th}^* Q$, $L^{\circ} \tr_{\r}^* Q$ and
$lg_{\b\m}(M^{\circ} \tr_{\b\m\r\th}^* Q) \geq lg(M \tr^* L)$. By
lemma \ref{simulation2}, there is a $R \in {\cal M}'$ such that
$L^{\circ} \tr_{\b\m\r\th}^* R$, $N^{\circ} \tr_{\r}^*R$ and
$lg_{\b\m}(L^{\circ} \tr_{\b\m\r\th}^* R) \geq 1$. Then, by lemma
\ref{diag*}, there is a $P \in {\cal M}'$ such that $Q
\tr_{\b\m\r\th}^* P$, $R \tr_{\r}^* P$ and $lg_{\b\m}(Q
\tr_{\b\m\r\th}^* P) \geq lg_{\b\m}(L^{\circ} \tr_{\b\m\r\th}^* R)
\geq 1$. Thus $M^{\circ} \tr_{\b\m\r\th}^* P$, $N^{\circ} \tr_{\r}^*
P$ and $lg_{\b\m}(M^{\circ} \tr_{\b\m\r\th}^* P) \geq lg(M \tr^* N)$.
\end{proof}

\begin{lemma}\label{fin}
Let $M \in {\cal M}''$ be such that $M^{\circ} \in
SN_{\b\m\r\th}$. Then $M \in SN$.
\end{lemma}

\begin{proof}
Since $M^{\circ} \in SN_{\b\m\r\th}$, let $n$ be the maximum of
$\tr_{\b\m}$ steps in the reductions of  $M^{\circ}$. If $M \not
\in SN$,  by lemma \ref{SNrth}, let $N$ be such that $M \tr^* N$
and $lg_{\b\m}(M \tr^* N) \geq n+1$.   By lemma \ref{simulation3},
there is $P$ such that $M^{\circ} \tr_{\b\m\r\th}^* P$ and
$lg_{\b\m}(M^{\circ} \tr_{\b\m\r\th}^* P) \geq lg_{\b\m}(M \tr^*
N)\geq n+1$.  Contradiction.
\end{proof}

\begin{theorem}\label{SN2}
Every typed $\l\m^{\f\et\ou}$-term is strongly normalizable.
\end{theorem}

\begin{proof}
A consequence of theorems \ref{SN_M}, \ref{retardemant} and lemmas
\ref{fin}, \ref{coding2}.
\end{proof}

\section{Recursive equations on types}\label{7}

We study here systems where equations on types are allowed. These
types are usually called recursive types. The subject reduction and
the decidability of type assignment are preserved but the strong
normalization may be lost. For example, with the equation $X = X \f
T$, the term $(\D \; \D)$ where $\D=\l x.(x \, x)$ is
typable but is not strongly normalizing. With the equation $X=X\f X$,
every term can be typed. By making some natural assumptions on the
recursive equations the strong normalization can be preserved. The
simplest condition is to accept the equation $X = F$ (where $F$ is a
type containing the variable $X$) only when the variable $X$ is
positive in $F$. For a set $\{X_i = F_i \;/ \; i \in I\}$ of mutually
recursive equations, Mendler \cite{Mend1} has given a very simple and
natural condition that ensures the strong normalization of the system.
He also showed that the given condition is necessary to have the
strong normalization.

Mendler's result concerns the implicative fragment of
intuitionistic logic. We extend here his result to full classical
logic. We now assume  ${\cal A}$ contains a specified subset
${\cal X}=\{X_i \; / \; i \in I \}$.

\begin{definition}\label{positive}
Let $X \in {\cal X}$. We define the subsets ${\cal P}^+(X)$ and
${\cal P}^-(X)$ of ${\cal T}$ (resp. ${\cal T}'$ ) as follows.
\begin{itemize}
\item $X \in {\cal P}^+(X)$
\item If $A \in ({\cal X}- \{X\})\cup \cal{A}$, then $A \in {\cal
P}^+(X) \cap {\cal P}^-(X)$.
\item If $A \in {\cal P}^-(X)$ and $B \in {\cal P}^+(X)$, then $A \f B
\in {\cal P}^+(X)$ and $B \f A \in {\cal P}^-(X)$.
\item If $A,B \in {\cal P}^+(X)$, then $A \et B, B \ou A \in {\cal
P}^+(X)$.
\item If $A,B \in {\cal P}^-(X)$, then $A \et B, B \ou A \in {\cal
P}^-(X)$.
\end{itemize}
\end{definition}

\begin{definition}
\begin{itemize}
\item Let ${\cal F}=\{F_i \; / \; i \in I \}$ be a set of types in
$\cal{T}$ (resp. in $\cal{T}'$). The congruence $\approx$ generated by
${\cal F}$ in $\cal{T}$ (resp. in $\cal{T}'$) is the least congruence
such that $X_i \approx F_i$ for each $i \in I$.
  \item We say that $\approx$ is good if, for each $X \in {\cal X}$,
if $X \approx A$, then $A \in {\cal P}^+(X)$.
\end{itemize}
\end{definition}

\subsection{Strong normalization of ${\cal S}^{\m}_{\approx}$}

Let $\approx$ be the congruence  generated by a set ${\cal F}$ of types
of $\cal{T}$.

\begin{theorem}[Mendler]\label{mendler}
If $\approx$ is good, then the system ${\cal S}_{\approx}$ is
strongly normalizing.
\end{theorem}

\begin{proof}
See \cite{Mend1} for the original proof and \cite{DaNo3} for an
arithmetical one.
\end{proof}

\begin{lemma}\label{coding3}
If $\G \v_{{\cal S}^{\m}_{\approx}} M : A$, then $\G^{\diamond}
\v_{{\cal S}^c_{\approx}} M^{\diamond} : A$.
\end{lemma}

\begin{proof}
By induction on the typing $\G \v_{{\cal S}^{\m}_{\approx}} M : A$.
\end{proof}

\begin{theorem}\label{mendler_mu}
If $\approx$ is good, then the system ${\cal S}^{\m}_{\approx}$ is
strongly normalizing.
\end{theorem}

\begin{proof}
Let $M \in {\cal M}'$ be a term typable in ${\cal
S}^{\m}_{\approx}$. By lemma \ref{sn1->sn2},  it is enough to show
that $M^{\diamond} \in SN_{\b}$. This follows immediately from
theorem \ref{mendler} and lemma \ref{coding3}.  Note that, in
\cite{DaNo3}, we also had given a direct proof of this result.
\end{proof}

\subsection{Strong normalization of ${\cal
    S}^{\f\et\ou}_{\approx}$}

Let ${\cal F}=\{F_i \; / \; i \in I \}$ be a set of types in ${\cal T}'$
and let ${\cal F}^{\circ}=\{F_i^{\circ} \; / \; i \in I \}$ be its
translation in ${\cal T}$. Let $\approx$ be the congruence
generated by ${\cal F}$  in ${\cal T}'$ and let $\approx^{\circ}$ be the
congruence generated by ${\cal F}^{\circ}$ in ${\cal T}$.

\begin{lemma}\label{good}
\begin{enumerate}
\item If $\approx$ is good, then
so is  $\approx^{\circ}$.
\item If $A \approx B$, then $A^{\circ} \approx^{\circ} B^{\circ}$.
\end{enumerate}
\end{lemma}

\begin{proof}
\begin{enumerate}
\item Just note that $A_1^{\circ}$ and $A_2^{\circ}$ are in positive
position in $\{A_1 \et A_2\}^{\circ}$ and $\{A_1 \ou A_2\}^{\circ}$.
\item By induction on the proof of $A \approx B$.
\end{enumerate}
\end{proof}

\begin{lemma}\label{coding4}
If $\G \v_{{\cal S}^{\f\et\ou}_{\approx}} M : A$, then $\G^{\circ}
\v_{{\cal S}^{\m}_{{\approx}^{\circ}}} M^{\circ} : A^{\circ}$.
\end{lemma}

\begin{proof}
By induction on a derivation of $\G \v_{{\cal S}^{\f\et\ou}_{\approx}}
M : A$.
\end{proof}

\begin{theorem}\label{mend_3}
If $\approx$ is good, then the system ${\cal
S}^{\f\et\ou}_{\approx}$ is strongly normalizing.
\end{theorem}

\begin{proof}
Let $M \in {\cal M}''$ be a term typable in ${\cal
S}^{\f\et\ou}_{\approx}$, then, by lemma \ref{coding4}, $M^{\circ}$ is
typable in ${\cal S}^{\m}_{{\approx}^{\circ}}$. Since, by lemma
\ref{good}, $\approx^{\circ}$ is good, then, by theorems
\ref{mendler_mu} and \ref{retardemant}, $M^{\circ} \in
SN_{\b\m\r\th}$, thus by lemma \ref{fin}, $M \in SN$.\\
\end{proof}

\noindent {\bf Remark}

Note that, in definition \ref{positive}, it was necessary to
define, for $X$ to be positive in a conjunction and a disjunction,
as being positive in both formulas since, otherwise, the previous
theorem will not be true as the following examples shows. Let
$A,B$ be any types. Note that, in particular,  $X$ may occur in
$A$ and $B$ and thus the negative occurrence  of $X$ in $X \f B$
is enough to get a non normalizing term.

\begin{itemize}
\item Let $F = A \et (X \f B)$ and $\approx$ be the congruence generated by
$X \approx F$. Let $M = \l x.((x \, \pi_2) \, x)$. Then $y : A \v_{{\cal
S}^{\f\et\ou}_{{\approx}}} (M \, \<y,M\>) : B$ and $(M \, \<y,M\>)
\not\in SN$ since it reduces to itself.

\item Let $G = A \ou (X \f B)$ and $\approx$ be the congruence generated by
$X \approx G$. Let $N = \l x (x \, [y.y,z.(z \, \o_2z)])$. Then
$\v_{{\cal S}^{\f\et\ou}_{{\approx}}} (N \, \o_2N) : B$ and $(N \,
\o_2N) \not\in SN$ since it reduces to itself.
\end{itemize}

\section{Appendix}\label{app}

\noindent {\bf Lemma \ref{simulation2}}
{\it Let $M \in {\cal M}''$. If $M \tr N$, then there is $P \in {\cal M}'$
such that $M^{\circ} \tr_{\b\m\r\th}^* P$, $N^{\circ} \tr_{\r}^*P$ and
$lg_{\b\m}(M^{\circ} \tr_{\b\m\r\th}^* P) \geq 1$}.

\begin{proof} We consider only the case of redexes.

\begin{itemize}
\item If $(\l x. M \; N) \tr M[x:=N]$, then

$\{(\l x. M \; N)\}^{\circ} = (\l x. M^{\circ} \; N^{\circ})
\tr_{\b}M^{\circ}[x:=N^{\circ}] = \{M[x:=N]\}^{\circ}$.

\item If $(\< M_1,M_2 \> \; \pi_i) \tr M_i$, then

$\{(\< M_1,M_2 \> \; \pi_i)\}^{\circ} =\m \a. (\vp \; (\l x. (x \;
M_1^{\circ} \; M_2^{\circ}) \; \l x_1.\l x_2.
 \m \g. (\a \;
 x_i)))$

 $ \tr_{\b}^+ \ \mu \a. (\vp \; \m \g. (\a \; M_i^{\circ})) \tr_{\r}\ \mu \a. (\a \; M_i^{\circ}) \tr_{\th} M_i^{\circ}$.

\item If $(\o_i M\; [x_1. N_1 ,x_2. N_2]) \tr N_i[x_i:=M]$, then

$\{(\o_i M\; [x_1. N_1 ,x_2. N_2])\}^{\circ} =$

$\m \a. (\vp \; (\l x_1. \l x_2. (x_i \; M^{\circ}) \; \l x_1. \m\g.(\a
\; N_1^{\circ})
  \; \l x_2. \m\g.(\a \; N_2^{\circ}))) $

$\tr_{\b}^+ \ \mu \a. (\vp \; \m\g.(\a \;
N_i^{\circ}[x_i:=M^{\circ}])) \tr_{\r}\mu \a. (\a \;
N_i^{\circ}[x_i:=M^{\circ}]) \tr_{\th} N_i^{\circ}[x_i:=M^{\circ}]
 $

$= \{N_i[x_i:=M]\}^{\circ}$.

\item If $(M \; [x_1.N_1 , x_2.N_2] \; N)
\tr (M \; [x_1.(N_1 \; N) , x_2.(N_2 \; N)])$, then

$\{(M \; [x_1.N_1 , x_2.N_2] \; N)\}^{\circ} =$

$(\m \a. (\vp \; (M^{\circ} \; \l x_1. \m\g.(\a \;
  N_1^{\circ}) \; \l x_2. \m\g.(\a \; N_2^{\circ}))) \; N^{\circ}) $

$\tr_{\m} \ \m \a. (\vp \; (M^{\circ} \; \l x_1. \m\g.(\a \;
  (N_1^{\circ} \; N^{\circ})) \; \l x_2. \m\g.(\a \; (N_2^{\circ} \;
  N^{\circ}))))$

$ = \m \a. (\vp \; (M^{\circ} \; \l x_1. \m\g.(\a \;
  \{(N_1 \; N)\}^{\circ}) \; \l x_2. \m\g.(\a \; \{(N_2 \;
  N)\}^{\circ})))$

$= \{(M \; [x_1.(N_1 \; N) , x_2.(N_2 \; N)])\}^{\circ}$.

\item If $(M \; [x_1.N_1 , x_2.N_2] \; \pi_i)
\tr (M \; [x_1.(N_1 \; \pi_i) , x_2.(N_2 \; \pi_i)])$, then

$\{(M \; [x_1.N_1 , x_2.N_2] \; \pi_i)\}^{\circ} =$

$\m \a. (\vp \; ( \m \b. (\vp \; (M^{\circ} \; \l x_1. \m\g.(\b \;
  N_1^{\circ}) \; \l x_2. \m\g.(\b \; N_2^{\circ})))  \; \l y_1.\l y_2. \m\g.(\a \;
  y_i))) $

$\tr_{\m} \ \m \a. (\vp \; \m \b. (\vp \; (M^{\circ} \; \l x_1.
\m\g.(\b \;
  (N_1^{\circ} \;  \l y_1.\l y_2. \m\g.(\a \;
  y_i))) \;$

$ \l x_2. \m\g.(\b \; (N_2^{\circ} \;  \l y_1.\l y_2. \m\g(\a \;
  y_i)))))) $

$\tr_{\r} \ \m \a. (\vp \; (M^{\circ} \; \l x_1. \m\g.(\vp \;
(N_1^{\circ} \; \l
  y_1.\l y_2. \m\g.(\a \; y_i))) \; $

$\l x_2. \m\g.(\vp \; (N_2^{\circ} \; \l
  y_1.\l y_2. \m\g.(\a \; y_i))))) = P$.

and $\{(M \; [x_1.(N_1 \; \pi_i) , x_2.(N_2 \;
\pi_i)])\}^{\circ}=$

$\m \b (\vp \; (M^{\circ} \; \l x_1 \m\g (\b \; \m \a (\vp \;
(N_1^{\circ} \; \l y_1 \l y_2 \m\g (\a \; y_i)))) \; $

$\l x_2 \m\g (\b \; \m \a (\vp \; (N_2^{\circ} \; \l y_1 \l y_2
\m\g (\a \; y_i)))) )) \tr_{\r}^+ P$.

\item If $(M \; [x_1.N_1 , x_2.N_2] \; [y_1.L_1 , y_2.L_2]) \tr$

$(M \; [x_1.(N_1 \; [y_1.L_1 , y_2.L_2]) , x_2.(N_2 \; [y_1.L_1 ,
y_2.L_2])])$, then

$\{(M \; [x_1.N_1 , x_2.N_2] \; [y_1.L_1 , y_2.L_2])\}^{\circ} =$

$\m \a. (\vp \; ( \m \b. (\vp \; (M^{\circ} \; \l x_1. \m\g.(\b \;
  N_1^{\circ}) \; \l x_2. \m\g.(\b \; N_2^{\circ}))) \; \l y_1. \m\g.(\a \;
  L_1^{\circ}) \; \l y_2. \m\g.(\a \; L_2^{\circ}))) $

$\tr_{\m} \ \m \a. (\vp \; \m \b. (\vp \; (M^{\circ} \; \l x_1.
\m\g.(\b \;
  (N_1^{\circ} \;  \l y_1. \m\g.(\a \;
  L_1^{\circ}) \; \l y_2. \m\g.(\a \; L_2^{\circ}))) \; $

$\l x_2. \m\g.(\b \; (N_2^{\circ} \;  \l y_1. \m\g.(\a \;
  L_1^{\circ}) \; \l y_2. \m\g.(\a \; L_2^{\circ})))))) $

$\tr_{\r} \ \m \a. (\vp \; (M^{\circ} \; \l x_1. \m\g.(\vp \;
  (N_1^{\circ} \;  \l y_1. \m\g.(\a \;
  L_1^{\circ}) \; \l y_2. \m\g.(\a \; L_2^{\circ})))) \; $

$\l x_2. \m\g.(\vp \; (N_2^{\circ} \;  \l y_1. \m\g.(\a \;
  L_1^{\circ}) \; \l y_2. \m\g.(\a \; L_2^{\circ})))) = P$.

and $\{(M \; [x_1.(N_1 \; [y_1.L_1 , y_2.L_2]) , x_2.(N_2 \;
[y_1.L_1 ,
    y_2.L_2])])\}^{\circ} = $

$\m \b. (\vp \; (M^{\circ} \; \l x_1. \m\g.(\b \;\m \a.
  (\vp \; (N_1^{\circ} \;  \l y_1. \m\g.(\a \;
  L_1^{\circ}) \; \l y_2. \m\g.(\a \; L_2^{\circ})))) \;$

$\l x_2. \m\g.(\b \; \m \a. (\vp \;(N_2^{\circ} \; \l y_1. \m\g.(\a \;
  L_1^{\circ}) \; \l y_2. \m\g.(\a \; L_2^{\circ})))))) \tr_{\r}^+ P$.

\item If $(\mu \a. M \; N) \tr \mu \a. M[(\a \; L):= (\a \; (L \; N))]$,
  then

$\{(\m \a. M \; N)\}^{\circ} = (\m \a. M^{\circ} \; N^{\circ})
\tr_{\m} \m \a.  M^{\circ}[(\a \; L^{\circ}) := (\a \; (L^{\circ}
\; N^{\circ}))]$

$ = \m \a. M^{\circ}[(\a \; L^{\circ}) := (\a \; \{(L \;
N)\}^{\circ})] = \{\m \a. M[(\a \; L) := (\a \; (L \;
N))]\}^{\circ}$.

\item If $(\m \b. M \; \pi_i) \tr \m \b. M[(\b \; N):= (\b \; (N \;
  \pi_i))]$, then

$\{(\m \b. M \; \pi_i)\}^{\circ} =  \m \a. (\vp \; (\m \b.
M^{\circ} \; \l x_1.\l
  x_2. \m \g. (\a \; x_i)))$

  $ \tr_{\m} \ \m \a. (\vp \; \mu \b.  M^{\circ}[(\b \; N^{\circ}) := (\b
\; (N^{\circ} \; \l
  x_1.\l x_2. \m\g(\a \; x_i)))]) $

$\tr_{\r} \ \m \a.  M^{\circ}[(\b \; N^{\circ}) := (\vp \;
(N^{\circ} \; \l x_1.\l x_2. \m\g.(\a \; x_i)))] = P$.

and $\{\m \b. M[(\b \; N):= (\b \; (N \; \pi_i))]\}^{\circ} =$

$\m \b. M^{\circ}[(\b \; N^{\circ}) := (\b \; \m \a. (\vp \; (N^{\circ}
\; \l x_1.\l x_2. \m\g.(\a \; x_i))))] \tr_{\r}^* P$.

\item If $(\mu \b. M \; [x_1. N_1 ,x_2. N_2]) \tr \mu \b. M[(\b \; N):=
  (\b \; (N \; [x_1. N_1 ,x_2. N_2]))]$, then

$\{(\mu \b. M \; [x_1. N_1 ,x_2. N_2])\}^{\circ} =$

$\m \a. (\vp \; (\m \b. M^{\circ} \; \l x_1. \m\g.(\a \;
N_1^{\circ}) \; \l x_2. \m\g.(\a \; N_2^{\circ})))$

$ \tr_{\m}^+ \ \m \a. (\vp \; \m \b.  M^{\circ}[(\b \;
N^{\circ}):= (\b \; (N^{\circ} \; \l x_1. \m\g.(\a \; N_1^{\circ})
\; \l x_2. \m\g.(\a \; N_2^{\circ})))]) $

$\tr_{\r} \ \m \a.  M^{\circ}[(\b \; N^{\circ}):= (\vp \;
(N^{\circ} \; \l x_1. \m\g.(\a \; N_1^{\circ}) \; \l x_2. \m\g.(\a
\; N_2^{\circ})))] = P$.

and $\{\mu \b. M[(\b \; N):= (\b \; (N \; [x_1. N_1
    ,x_2. N_2]))]\}^{\circ} =$

$\m \b. M^{\circ}[(\b \; N^{\circ}):= (\b \; \m \a. (\vp \; (N^{\circ}
\; \l x_1. \m\g.(\a \; N_1^{\circ}) \; \l x_2. \m\g.(\a \;
N_2^{\circ}))))] \tr_{\r}^* P$.

\end{itemize}
\end{proof}

\end{document}